\documentclass{amsart}
\usepackage{amssymb}

\usepackage{mathrsfs}

\usepackage{amscd} \usepackage{amsmath}

\usepackage{amsthm} \usepackage{amssymb}

\usepackage[all]{xy}

\newtheorem{lemma}{Lemma}

\newtheorem{proposition}[lemma]{Proposition}

\newtheorem{corollary}[lemma]{Corollary}

\newtheorem{theorem}[lemma]{Theorem}

\newtheorem{prob}{Problem}
\DeclareFontEncoding{OT2}{}{} % to enable usage of cyrillic fonts
  \newcommand{\textcyr}[1]{%
    {\fontencoding{OT2}\fontfamily{wncyr}\fontseries{m}\fontshape{n}%
     \selectfont #1}}
\newcommand{\Sha}{{\mbox{\textcyr{Sh}}}}

\newcommand{\F}{\ensuremath{\mathbb F}}

\newcommand{\PP}{\mathbb{P}}

\newcommand{\N}{\ensuremath{\mathbb N}}
\newcommand{\Q}{\ensuremath{\mathbb Q}}

\newcommand{\Z}{\ensuremath{\mathbb Z}}

\newcommand{\C}{\ensuremath{\mathbb C}}

\newcommand{\ra}{\ensuremath{\rightarrow}}

\newcommand{\Hom}{\operatorname{Hom}}

\newcommand{\Ker}{\operatorname{Ker}}

\newcommand{\unr}{\operatorname{unr}}

\newcommand{\Gal}{\operatorname{Gal}}
\newcommand{\Aut}{\operatorname{Aut}}

\newcommand{\OO}{\mathcal{O}}

\newcommand{\Gm}{\mathbb{G}_m}
\newcommand{\Pic}{\operatorname{Pic}}
\newcommand{\FPic}{\textbf{Pic}}

\newcommand{\Br}{\operatorname{Br}}

\newcommand{\Li}{\operatorname{Li}}

\newcommand{\sG}{\mathcal{G}}

\newcommand{\sS}{\mathscr{S}}

\newcommand{\sm}{\ensuremath{\mathfrak{m}}}

\renewcommand{\sp}{\ensuremath{\mathfrak{p}}}

\renewcommand{\Q}{\ensuremath{\mathbb{Q}}}

\renewcommand{\Z}{\ensuremath{\mathbb{Z}}}

\newcommand{\Pro}{\ensuremath{\mathbb{P}}}

\newcommand{\Kb}{\overline{K}}

\renewcommand{\C}{\ensuremath{\mathbb{C}}}

\renewcommand{\F}{\ensuremath{\mathbb{F}}}

\newcommand{\G}{\ensuremath{\mathbb{G}}}

\newcommand{\cbd}{c}

\DeclareMathOperator{\Nm}{Nm}

\DeclareMathOperator{\Hp}{H}

\DeclareMathOperator{\Gl}{GL} 
\DeclareMathOperator{\Sl}{SL}

\DeclareMathOperator{\Pgl}{PGL}

\DeclareMathOperator{\Ob}{\Delta} 
\renewcommand{\Br}{\operatorname{Br}}

\DeclareMathOperator{\res}{res}

\renewcommand{\Pic}{\operatorname{Pic}}

\renewcommand{\FPic}{\bold{Pic}}

\renewcommand{\ra}{\rightarrow}

\DeclareMathOperator{\cores}{cores}

\renewcommand{\Gm}{\mathbb{G}_m}

\renewcommand{\PP}{\mathbb{P}}

\renewcommand{\Im}{\operatorname{Im}}

\renewcommand{\Ker}{\operatorname{Ker}}

\DeclareMathOperator{\Jac}{\operatorname{Jac}}

\newcommand{\supp}{\operatorname{supp}}

\newcommand{\Sel}{\operatorname{Sel}}

\renewcommand{\Li}{\operatorname{Li}}

\renewcommand{\OO}{\mathfrak{o}}

\renewcommand{\unr}{\operatorname{unr}}

\newcommand{\kummer}[2]{({#1}^\times/{#1}^{\times{#2}})^2}

\newcommand{\Pstar}{P^*}

\email{pete@math.uga.edu}
\email{sharif@math.duke.edu}

\title{Period, Index and Potential Sha}
\author{Pete L. Clark}
\thanks{The first author is partially supported by National Science Foundation grant DMS-0701771}
\author{Shahed Sharif}

\begin{document}
\maketitle

\begin{abstract}
In this paper we advance the theory of O'Neil's period-index obstruction map and derive consequences for 
the arithmetic of genus one curves over global fields.  Our first result implies that for every pair of positive integers 
$(P,I)$ with $P \ | \ I \ | \ P^2$, there exists a number field $K$ and a genus one curve $C_{/K}$ with period $P$ 
and index $I$.  Second, let $E_{/K}$ be any elliptic curve over a global field $K$, and let $P > 1$ be any integer 
indivisible by the characteristic of $K$.  We construct infinitely many genus one curves $C_{/K}$ with period $P$, 
index $P^2$, and Jacobian $E$.  We deduce strong consequences on the structure of Sharevich-Tate 
groups under field extension.
\end{abstract}

\tableofcontents

\section{Introduction}

\subsection{Notation and conventions} \textbf{} \\ \\ \noindent Throughout the paper $K$ shall denote a global field ---
i.e., a finite field extension of either $\Q$ or $\F_p(T)$ --- and
$E$ shall denote an elliptic curve defined over $K$. 
\\ \\
Let $P$ be a positive integer which is \emph{not} divisible by the
characteristic of $K$.  We define $\Pstar$ to be $P$ if $P$ is
odd and $2P$ if $P$ is even.
\\ \\
Let $\overline{K}$ denote a fixed \emph{separable} closure of $K$, and let $\mathfrak{g}_K = \Aut(\overline{K}/K)$ be the absolute 
Galois group of $K$.
\\ \\
We abbreviate the Galois cohomology group
$\Hp^1(\mathfrak{g}_K,E(\overline{K}))$ to $\Hp^1(K,E)$ and
call it the \textbf{Weil-Ch\^atelet group} of $E$ over $K$.  Recall that this is a torsion abelian group.  \\ \\
The letter $\eta$ shall denote an element of $\Hp^1(K,E)$.  Such classes
$\eta$ are in canonical bijection with the set of pairs
$(C,\iota)$, where $C_{/K}$ is a genus one curve and $\iota:
\FPic^0(C) \ra E$ is an isomorphism from the Albanese/Picard
variety of $C$ to $E$.  In other words, $\iota$ endows $C$ with
the structure of a principal homogeneous space (or torsor) under
$E$.  It follows that $C_{/K}$ itself determines, and is
determined by, an orbit of $\Aut(E)$ on $H^1(K,E)$.  
\\ \\
The \textbf{period} of $\eta \in H^1(K,E)$ is its order in the group.  
In terms of the corresponding torsor $(C,\iota)$, the period is the least positive degree 
of a $K$-rational divisor class on $C$.  The \textbf{index} of $\eta$ is the gcd over all 
degrees $[L:K]$ of field extensions $L/K$ such that the restriction of $\eta$ to $H^1(L,E)$ 
is trivial.  In terms of $(C,\iota)$, the index is the least degree of a $K$-rational divisor.  
By Riemann-Roch, it is also the least degree of an extension $L/K$ such that $C$ has an $L$-rational 
point.
\\ \indent
Notice that both the period and the index of $(C,\iota)$ depend only on the underlying 
curve $C$.  Therefore no harm will come from the abuse of language ``the cohomology class $\eta$ corresponding
to $C_{/K}$,'' and we shall use this simplified language in the sequel.
\\ \\
We denote by $\Sigma_K$ the set of all places of $K$
(including Archimedean places in the number field case). For a
place $v$ of $K$, we denote the image of a class $\eta \in
\Hp^1(K,E)$ under the local restriction map $H^1(K,E) \ra
H^1(K_v,E)$ by $\eta_v$. In geometric terms, $\eta_v$ is just the
base extension of the curve (or rather, the principal homogeneous
space$\ldots$) $C$ from $K$ to $K_v$.  By the \textbf{support} of
a class we mean the finite set of $v \in \Sigma_K$ such that
$\eta_v \neq 0$.  The classes $\eta$ with empty support form a
subgroup $\Sha(K,E)$, the \textbf{Shafarevich-Tate group} of $E_{/K}$.

\subsection{Statement of the main results}
\begin{theorem}
\label{MT1} Suppose $\#E(K)[\Pstar] = (\Pstar)^2$.  Then, for any
positive integer $D \ | \ P$, there are infinitely many classes
$\eta \in \Hp^1(K,E)$ of period $P$ and index $P \cdot D$. These
classes can be chosen so as to be locally trivial except possibly
at two places of $K$.
\end{theorem}
\noindent 

\begin{theorem}

\label{MT2} Let $E_{/K}$ be an elliptic curve and $S_K \subset
\Sigma_K$ a finite set of places of $K$.  There exists an infinite
sequence $\{\eta_i\}_{i=0}^{\infty}$ of elements of $\Hp^1(K,E)$ such that: \\
$\bullet$ $\eta_0 = 0$. \\
$\bullet$ For all $v \in S_K$ and all $i \in \N$, $\res_v \eta_i = 0$. \\
$\bullet$ For all $i, \ j \in \N$ with $i \neq j$, $\eta_i - \eta_j$ has
period $P$ and index $P^2$.
\end{theorem}

\begin{theorem}\label{MT3} 
For any positive integer $r$, there exists a degree
$P$ field extension $L/K$ such that $\Sha(L,E)$ contains at least
$r$ elements of order $P$.
\end{theorem}

\subsection{Discussion of the results} \textbf{} \\ \\ \noindent
Let $C$ be a genus one curve over an arbitrary field $K$.  It is well known (e.g, \cite[Cor. 13]{WCII})
that the period $P$ and the index $I$ of $C$ satisfy the divisibilities 
\begin{equation}
\label{DIVEQ} P \ | \ I \ | \ P^2. 
\end{equation}
In their seminal 1958 paper \cite{LT}, Lang and Tate showed that for any pair $(P,I)$ of 
positive integers satisfying (\ref{DIVEQ}), there exists a genus one curve $C$ defined over the 
iterated Laurent series field $\C((t_1))((t_2))$ with period $P$ and index $I$.  
\\ \\
This raises the question of the possible values of $P$ and $I$ for genus one curves over a 
local or global field.  Lichtenbaum \cite{Lichtenbaum} showed that $P = I$ for every genus one curve 
over a nondiscrete, locally compact field.\footnote{More precisely, Lichtenbaum proved this under 
the assumption that $P$ is not divisible by the characteristic of $K$ -- the same assumption which is in force 
for us -- but Milne later extended Tate's local duality theory to this case \cite{Milne} and accordingly was able to remove this 
hypothesis.} 
\\ \\
Suppose $K$ is a field which admits at least one degree $P$ cyclic extension and such that there exists an elliptic curve 
$E_{/K}$ with full $P$-torsion: $\# E[P](K) = P^2$.  Then Lang and Tate were able to show that there exists a class 
$\eta \in H^1(K,E)$ with period and index both equal to $P$.  
\\ \\
Let us assume henceforth that $K$ is a global field.  In this case, the argument of Lang and Tate readily yields 
the fact that $\eta$ may be taken to have support at at most one place of $K$.
\\ \\
Conversely, Cassels \cite[Theorem 1.3]{Cassels} showed that $I = P$ for
classes with empty support.  Moreover $I = P$ for classes whose support has cardinality one, as was 
first shown by L. Olson \cite[Thm. 15]{Olson} and more recently ``rediscovered'' by the first author 
\cite[Prop. 6]{Crelle}.
\\ \\
The first examples of genus one curves over a global field with $I > P$ are due to Cassels \cite{CasselsV}, who found 
examples over $K = \Q$ with $P = 2$, $I = 4$.  Cassels' examples are closely related to the theory of explicit 
$2$-descent, a connection which is reconsidered in a forthcoming work of the first author \cite{P2}.  More 
recently, the first author constructed, for any prime number $p$, classes $\eta$ with $P = p, \ I = p^2$ in the Weil-Chat\^elet 
group of any elliptic curve $E_{/K}$ over a number field with full $p$-torsion \cite[Theorem 3]{WCI}.  The method crucially employs 
a period-index obstruction map due to C.H. O'Neil \cite{O'Neil}.
\\ \\
Our Theorem \ref{MT1} is therefore to be viewed as a substantial generalization of \cite[Theorem 3]{WCI}.  In particular, 
we now know that any pair $(P,I)$ satisfying (\ref{DIVEQ}) arises as the period and index of a genus one curve defined over 
some number field (depending on $P$).  Moreover, the fact that we can construct such classes which are supported at two 
places is, in view of the aforementioned results of Cassels and Olson, optimal, and answers a question raised by M. \c Ciperiani.
\\ \\
Having established Theorem \ref{MT1}, we naturally wish to understand the possible values of period and
index for genus one curves defined over a \emph{fixed} global field $K$,
or---better yet---inside the Weil-Ch\^atelet group $\Hp^1(K,E)$
of a fixed elliptic curve $E_{/K}$. 
\\ \\ \noindent
Our Theorem \ref{MT2} shows that for any elliptic curve $E$ over a global
field $K$ and any $P > 1$ indivisible by the
characteristic of $K$, there exist infinitely many genus one
curves with period $P$, index $P^2$ and Jacobian $E$.  Of course the statement of 
Theorem \ref{MT2} is significantly more complicated than this, and its significance is probably hard 
to appreciate.  However, we need this precise statement, especially the ``difference
properties'' of the sequence $\{\eta_i\}$, in the proof of Theorem \ref{MT3}.  
\\ \\
In order to place Theorem \ref{MT3} into context, let us again recall some prior results, this time on the problem of 
constructing ``large Shafarevich-Tate groups.''  More precisely, we fix a global field $K$, an integer $P > 1$ 
and a positive integer $r$, and the goal is prove the existence of an elliptic curve $E_{/K}$ whose Shafarevich-Tate 
group $\Sha(K,E)$ contains at least $r$ elements of order $P$.
\\ \\
The first results here are due to Cassels \cite{CasselsVI}, who in 1984 solved the aforementioned problem 
for $K = \Q$ and $P = 3$.  (This was also the first proof of the weaker fact that $\Sha(\Q,E)$ is unbounded as $E$ 
ranges over all elliptic curves $E_{/\Q}$.)  Cassels' examples all have $j = 0$ and exploit the extra structure on such curves afforded by 
the existence of an order $3$ automorphism.  The problem has also been solved for $P = 2$ by B\"olling
\cite{Bolling}, and for $P = 5$ by Fischer \cite{Fischer}.  In his 2003 Georgia PhD thesis, Steve Donnelly established 
the result for $P = 7$.  Among prime values of $P$, this is a transitional case: the modular curve $X(P)$ has genus 
$0$ precisely for $P = 2,3,5$, a phenomenon which the aforementioned proofs implicitly take advantage of.  Now $X(7)$ is 
Klein's quartic curve (of genus $3$) but at least $X_1(7)$ still has genus zero.  For prime $P > 7$ no elliptic curve 
$E_{/\Q}$ has a rational $P$-torsion point, a difficulty which seems insurmountable by present methods.  
\\ \\
So, reasonably, there has also been some work showing that either
the $p$-Selmer group $\Sel^p(K,E)$ or $\Sha(K,E)[p]$ can be made
arbitrarily large when one varies over all elliptic curves $E$
defined over number fields $K$ whose degree $[K:\Q]$ is bounded by
a certain function of $P$.  Notably, R. Kloosterman and E. Schaefer
showed \cite{KS} that $\dim_{\F_p} \Sel^p(K,E)$ is unbounded as
$K$ ranges over all field extensions $K/\Q$ of degree $f_1(p) =
O(p)$; later Kloosterman showed \cite{Kloosterman} that
$\dim_{\F_p} \Sha(K,E)[p]$ is unbounded as $K$ ranges extensions
of degree $f_2(p) = O(p^4)$.
\\ \\
In \cite[Thm. 1]{WCI}, the first author showed that
if $\#E(K)[p] = p^2$, $\Sha(L,E)[p]$ is unbounded as $L$ ranges
over all degree $p$ field extensions.  The argument can be applied
to any elliptic curve defined over a global field (of
characteristic not divisible by $p$) at the cost of first
trivializing the Galois action on the $p$-torsion.  We deduced
that, for every $E_{/K}$, $\Sha(L,E)[p]$ is unbounded as $L$
ranges over extensions of degree at most $f_3(p) = p(p^2-1)(p^2-p)
\leq p^5$.  Moreover, upon restricting to elliptic curves with
potential complex multiplication, one gets the bound $f_4(p) \leq 2p^3$.  
\\ \\
In contrast, our Theorem \ref{MT3} extends the bound $[L:K] = P$ of \cite[Thm. 1]{WCI} 
to \emph{all} elliptic curves and all integers $P > 1$.  An interesting 
question (which we are not able to answer) is whether Theorem \ref{MT3} is in fact the optimal result of its kind.   

\subsection{Remarks on prior individual work} \textbf{} \\ \\ \noindent
Each of the authors did substantial work on the period-index problem for genus 
one curves before entering into this collaboration.  But whereas the first author's prior work has already been published 
\cite{WCI}, \cite{WCII}, \cite{Crelle}, the second author's work was done as part of his 2006 Berkeley thesis \cite{Sharif}.  
Upon reading \cite{Sharif}, the first author saw the prospect for some additional improvements, at which point the 
collaboration began.  The present paper thus includes both work of the second author's thesis as well as some further results 
which were obtained in collaboration.  The first author wishes to make sure that the second author's innovative and 
technically powerful contributions receive their due credit, so we have decided to depart from usual practice and be 
rather specific about the individual contributions.
\\ \indent
The first statement of Theorem~\ref{MT1} appears as \cite[Theorem
3]{WCI} under the additional assumption that $P$ is prime. The
general case of Theorem \ref{MT1} appears \cite[Theorem 4.2]{Sharif}.  Moreover, in
\cite{Sharif} the second author developed new techniques to
circumvent the rationality of the $P$-torsion and was able to give
examples of $I = P^2$ over $\Q$ for all odd $P$.  Theorem \ref{MT2} is the heart of 
the collaboration (as well as the paper): the first author supplied the statement and some 
strategic suggestions, whereas the argument itself was supplied by the second author, roughly along 
the lines of the special case appearing in \cite{Sharif}.  The deduction
of Theorem \ref{MT3} from Theorem \ref{MT2} is due to the first author.
\subsection{Organization of the paper} \textbf{} \\ \\ \noindent
We assume some familiarity with the literature on the
period-index problem, especially \cite{O'Neil} and \cite{WCI}; nevertheless, we 
begin with a brief review of the period-index obstruction map, and then go on 
to discuss some new ideas and techniques.  The first key point is a clarification
of the relationship between O'Neil's obstruction map $\Delta$ and
the quantity $I/P$. Whereas before it had been implicit in
\cite{O'Neil} (and explicit in \cite{WCI}) that one can use
$\Delta$ to determine whether or not $I = P$, here we present a
simple characterization of $I/P$ in terms of the obstruction to a
rational divisor class being represented by a rational divisor.
We also return to the point of the explicit computation of
O'Neil's obstruction map in the case full level $N$ structure for
even $N$. These matters are detailed in Section 2.
\\ \\
In Section 3 we give the proofs of Theorems 1, 2 and 3.
\\ \\
In Section 4, we survey what remains to be done on
the period-index problem for curves of genus one, and formulate
several open problems.

\section{On the period-index obstruction map} \noindent
In this section $K$ is an arbitrary field, $E_{/K}$ is an elliptic
curve, and $P$ is a positive integer not divisible by the
characteristic of $K$.  These hypotheses ensure that the finite
flat $K$-group scheme $E[P]$ is \'etale, so may be viewed as a
$\mathfrak{g}_K$-module.

\subsection{Three aspects of the period-index obstruction map}

\label{3ASPECTS}

The object of our affections is the \textbf{period-index
obstruction map}

\[\Delta_P: \Hp^1(K,E[P]) \ra \Br(K). \]
It can be defined in three different ways (and much of its utility comes from passage between the various definitions),
as we now recall (cf. \cite{O'Neil}, \cite{WCI}, \cite{WCII}).
\\ \\
1) For any ample line bundle $L$ on an abelian variety $A_{/K}$,
the functor $\sG_L$ which associates to a $K$-scheme $S$ the group
of all isomorphisms $(x,\psi): L_{/S} \stackrel{\sim}{\ra}
\tau_x^*(L_{/S})$ between $L_{/S}$ and one of its translates is
represented by a $K$-group scheme, Mumford's \textbf{theta group}.
The subgroup of automorphisms of $L$ gives rise to an embedding
$\Gm \hookrightarrow \sG_L$.  The quotient is canonically
isomorphic to $\kappa(L)$, the kernel of the canonical
homomorphism \[\varphi_L: A \ra A^{\vee}, x \mapsto \tau_x^*(L)
\otimes L^{-1}. \]  Here $A$ will be an elliptic curve and $L$ will
be the line bundle associated to the divisor $P[O]$ on $E$; then
$\kappa(L) = E[P]$.

\begin{proposition}
For $n\geq 2$ we have the following commutative diagram of group schemes:
\begin{equation} \label{eq:theta}
\xymatrix{0 \ar[r] & \G_m \ar[r] \ar[d] & \sG_L \ar[r] \ar[d] & E[P] \ar[r] \ar[d] & 0 \\
0 \ar[r] & \G_m \ar[r] & \Gl_P \ar[r] & \Pgl_P \ar[r] & 0 }
\end{equation}
\end{proposition}

\begin{proof}
This is Proposition 2.1 in \cite{O'Neil}. For our purposes, we
will only need to know the vertical map on the right.  
We view $E[P]$ as an automorphism group for diagrams $E\rightarrow
\Pro^{P-1}$ --- that is, an element of $E[P]$ acts on the global
sections of the line bundle $\mathcal{L}(P[O])$, and thus induces
an automorphism of $\Pro^{P-1}$. This gives an element of $PGL_P$
as required.
\end{proof}
\noindent
The machinery of nonabelian Galois cohomology \cite{CL} supplies a
connecting map from $\Hp^1(K,E[P]) \ra \Hp^2(K,\Gm)$. After
identifying the latter with $\Br(K)$, this gives our first
definition of $\Delta_P$.
\\ \\
2) On any nonsingular, complete, geometrically integral variety $V_{/K}$ there is an
exact sequence (e.g. \cite[$\S 9.1$]{BLR})

\begin{equation}
\label{EQONE}
0 \ra \Pic(V) \ra \FPic(V)(K) \stackrel{\delta_V}{\ra}
\Br(K) \stackrel{\gamma}{\ra} \Br(V). \end{equation} 
In particular, given a $K$-rational divisor class $D$ on $V$, the
obstruction to $V$ being represented by a $K$-rational divisor is
an element of $\Br(K)$.  A Galois descent argument (e.g. \cite[Prop. 28]{WCII}) shows that
$\Hp^1(K,E[P])$ classifies pairs $(C,D)$---where $C \in
\Hp^1(K,E)$ and $D \in \FPic^P(C)(K)$ is a $K$-rational divisor
class---modulo the relation $(C,D) \sim (C',D')$ if there exists
an isomorphism of torsors $f: C \ra C'$ with $f^*D' = D$.  One may
then define
\[\Delta_P((C,D)) = \delta_C(D). \]
3) On the other hand, $\Hp^1(K,E[P])$ classifies $K$-morphisms $\varphi: C \ra V$, where $C \in \Hp^1(K,E)$ and $V$ is a twisted form of $\PP^{P-1}$.  We may
then define $\Delta_P(\varphi: C \ra V)$ as the class of $V$ in $\Br(K)$.
\\ \\
It follows from 3) that $\Delta_P(\Hp^1(K,E[P]))$ consists of
elements of $\Br(K)$ whose \emph{index} divides $P$; \emph{a
fortiori} we have the important relation
\[\Delta_P(\Hp^1(K,E[P])) \subset \Br(K)[P]. \]

\subsection{Lichtenbaum-Tate Duality} \textbf{} \\ \\ \noindent
As above, we let $E$ be an elliptic curve defined over an arbitrary field $K$, and now let 
$n$ denote a positive integer indivisible by the characteristic of $K$.\footnote{Thus $n$ satisfies exactly the same 
requirements as our ``fixed''' positive integer $P$.  The merit of considering both ``fixed $P$'' and ``variable $n$'' 
will become clear in the next section.}  We have the \textbf{Kummer sequence}
\begin{equation}
\label{KUMMEREQ}
0 \ra E(K)/nE(K) \stackrel{\iota}{\ra} H^1(K,E[n]) \ra H^1(K,E)[n] \ra 0. 
\end{equation}
Using $\iota$ and $\Delta$, we define a map
$\Li: \Hp^1(K,E[n]) \times E(K) \ra \Br(K)$,
\[\Li(\xi,x) = \Delta(\xi+\iota(x)) - \Delta(\xi) -
\Delta(\iota(x)). \] 
Since $\Delta(\iota(E(K)/nE(K))) = 0$, $\Li$ depends only on the image of $\xi$ in $\Hp^1(K,E)[n]$ and
on the image of $x$ in $E(K)/nE(K)$, i.e., it descends to give a map

\begin{equation} \label{eq:obtate}
\Li: \Hp^1(K,E)[n] \times E(K)/nE(K) \ra \Br(K)[n].
\end{equation}

\begin{theorem}(Lichtenbaum)
The map $\Li$ coincides with the Tate pairing $T$.
\end{theorem}
\noindent This has two immediate, and important, consequences.
First, since $T$ is bilinear, so is $\Li$, and this means (by
definition) that $\Delta$ itself is a quadratic map.  Secondly, if $K$ is
complete, discretely valued, and with finite residue field, then
$\Br(K)[n] = (\frac{1}{n}\Z)/\Z$, and $\Li$ puts the finite abelian
groups $\Hp^1(K,E)[n]$ and $E(K)/PE(n)$ in Pontrjagin
duality.\footnote{The equality of period and index in this context
follows almost immediately~\cite{Lichtenbaum2}.}

\subsection{Theta functoriality} \textbf{} \\ \\ \noindent
Let $\eta$ be a class in $H^1(K,E)[n]$.  By a \textbf{Kummer lift} of $\eta$ we mean a 
class $\xi \in H^1(K,E[n])$ whose image under the canonical map $H^1(K,E[n]) \ra H^1(K,E)[n]$ 
is $\eta$.  Of course, the exactness of the Kummer sequence (\ref{KUMMEREQ}) means that $\eta$ 
has at least one Kummer lift.  Following O'Neil and Clark, we attempt to use the obstruction maps 
$\Delta$ to study the discrepancy between the period and the index of $\eta$.
\\ \\ \noindent
However, in \cite{WCI} we only considered the case where $n$ is equal to the period $P$ of $\eta$.  But certainly we can 
also choose Kummer lifts $\xi_n \in H^1(K,E[n])$
whenever $n$ is any multiple of the period of $\eta$, and it turns
out to be quite useful to do so, and in particular to compare
various obstruction maps $\Delta_n$ of differing levels.
Geometrically speaking this amounts to considering along with the
theta group $\sG_L$ of our fixed line bundle $L = L(P[O])$ the
theta groups of all tensor powers $L^{n}$ of $L$ and various
natural homomorphisms between them.  The study of such
homomorphisms is indeed an integral part of Mumford's theory.
\\ \\
So let $m$ be yet another positive integer indivisible by the characteristic
of $K$. The natural inclusion $E[P] \hookrightarrow E[mP]$ of
$\mathfrak{g}_K$-modules induces a map \[j_m: \Hp^1(K,E[P]) \ra
\Hp^1(K,E[mP]). \]  Under the interpretation (2) of
$\Hp^1(K,E[N])$ as equivalence classes of pairs $(C,D)$, where $C
\in H^1(K,E)$ and $D \in \FPic^N(C)$, $j_m$ is the map $(C,D) \mapsto (C,mD)$. \\
Similarly, multiplication by $m$ induces a map
\[[m]: \Hp^1(K,E[mP]) \to \Hp^1(K,E[P]). \]
\begin{proposition}
\label{SHARPROP} If $\xi \in \Hp^1(K,E[P])$ and $\eta\in
\Hp^1(K,E[mP])$, then: \\
a) $\Delta_{mP} j(\xi) = m\Delta_P(\xi)$, and \\
b) $m\Ob_{mP} \eta = \Ob_P([m]\eta)$.
\end{proposition}

\begin{proof}
Mumford shows \cite[p. 309--310]{Mumford} that both $j$ and $[m]$
extend to morphisms of the theta group sequences: \\ \\

\[\xymatrix{0 \ar[r] & \G_m \ar[d]^{[m]} \ar[r] & \sG_L \ar[d]^{\epsilon_m} \ar[r] & E[P] \ar[r] \ar[d]^j & 0 \\
0 \ar[r] & \G_m \ar[r] & \sG_{L^m} \ar[r] & E[mP] \ar[r] & 0} \]

and

\[ \xymatrix{0 \ar[r] & \G_m \ar[d]^{[m]} \ar[r] & \sG_{L^m} \ar[d]^{\eta_m} \ar[r] & E[mP] \ar[r] \ar[d]^{[m]} & 0 \\
0 \ar[r] & \G_m \ar[r] & \sG_{L} \ar[r] & E[P] \ar[r] & 0}.\]
In each case the restriction to $\Gm$ is simply the $m$th power
map.  We remark that the map $\epsilon_m: \sG_L \ra
\sG_{L^m}$ is relatively straightforward to define: an isomorphism
$\psi: L \stackrel{\sim}{\ra} \tau_x^*L$ induces, by passage to
the $m$th power, a canonical isomorphism $\psi^{\otimes m}: L^m
\stackrel{\sim}{\ra} \tau_x^*(L^m)$, so $\epsilon_m: (x,\psi)
\mapsto (x,\psi^m)$.  These commutative ladders induce commutative ladders in nonabelian
Galois cohomology, and the commutativity of these last two
diagrams gives the desired result.

\end{proof}

\subsection{Applications to the quantity $I/P$} \textbf{} \\ \\
\noindent We begin with the following result, which was known to
O'Neil:

\begin{proposition}(\cite[Theorem 5]{WCI})
\label{EASYPROP} Let $E_{/K}$ be an elliptic curve over a field
$K$, and $P$ a positive integer indivisible by the characteristic
of $K$. Let $\eta \in \Hp^1(K,E)$ be of period $P$.  The following
are equivalent: \\
a) $\eta$ has index $P$. \\
b) There exists some Kummer lift $\xi$ of $\eta$ such that $\Delta_P(\xi) = 0$.
\end{proposition} % Extend this proposition to get when $\Ob \neq 0$

\begin{proof}

Indeed, in light of the second definition of $\Delta_P$, both
conditions express the fact that $C$ admits a rational divisor of
degree $P$.

\end{proof}

\noindent We are therefore interested in the remaining case in
which $\Delta_P(\xi) \neq 0$ for every Kummer lift $\xi$ of
$\eta$.
\\ \\
Let $C_{/K}$ be a curve of any genus, of period $P$ and index $I$.
Referring back to (\ref{EQONE}), we may define the \textbf{relative 
Brauer group} $\kappa(C/K) = \Im(\delta) = \Ker(\gamma)$. For any $n \in
\Z$, define moreover $\kappa^n(C/K) = \delta_C(\FPic^n(C)(K))$.
\begin{proposition}
The quotient $\kappa(C/K)/\kappa^0(C/K)$ is cyclic of order $I/P$.
\end{proposition}
\noindent
This is a reasonably well-known result -- c.f. \cite[Thm. 2.1.1]{CK}, \cite[Prop. 24]{WCII}
-- the standard proof of which employs a snake lemma argument.  But the following proof offers
some additional insight.

\begin{proof}
By definition of $P$ we have $\FPic^n(C)(K) = \emptyset$ unless
$n$ is a multiple of $P$, so \[\kappa(C/K) = \delta_C(\FPic(C)(K)) =
\delta_C(\bigcup_{n \in \Z} \FPic^{nP}(C)(K)) \]
\[= \bigcup_{n \in \Z} \delta(\FPic^{nP}(C)(K)) = \bigcup_{n \in \Z}
\kappa^{nP}(C/K).\]
Choose a rational divisor class $D$ of degree $P$; this in turn
determines a rational divisor class of each degree $nP$, namely
$D_{nP} = nD$. Put $\alpha = \delta_C(D)$, so that $\delta_C(D_{nP}) =
n\alpha$. Adding $D_{nP}$ induces a bijection of sets
$\FPic^0(C)(K) \ra \FPic^{nP}(C)$, and exhibits
\[\kappa^{nP}(C/K) = n\alpha + \kappa^0(C/K) \]
as a coset of the subgroup $\kappa^0(C/K)$ of $\Br(K)$.  This
shows that $\kappa(C/K)$ is the subgroup generated by $\alpha$ and
$\kappa^0(C/K)$.  Moreover, $C$ admits a rational divisor of
degree $nP$ if and only if $0 \in \kappa^{nP}(C/K)$ if and onlf if $n\alpha \in
\kappa^0(C/K)$.  The quantity $I/P$ is the least such value of
$n$, i.e., the order of

\[ \langle \alpha + \kappa^0(C/K) \rangle / \kappa^0(C/K) =
\kappa(C/K)/\kappa^0(C/K). \]

\end{proof}

\begin{proposition}
Let $\eta \in H^1(K,E)$ be a class with period $P$ and index $I$,
and let $\xi$ be any Kummer lift of $\eta$. Then
\begin{equation} \label{FUNDINEQ}
I/P \leq \min_{x \in E(K)/PE(K)} \# \Delta_P(\xi+x).
\end{equation}
\end{proposition}

\noindent

%\noindent We will say that two elements $\alpha, \ \beta$ in an

%abelian group $A$ are \emph{independent} if both are nonzero and

%they generate disjoint subgroups: $\langle \alpha \rangle \cap

%\langle \beta \rangle = 0$.

%

%\begin{theorem}

%Let $E_{/K}$ be an elliptic curve, and $C \in \Hp^1(K,E)$ be a

%genus one curve with period $P$, and $\xi \in \Hp^1(K,E[P])$ a

%Kummer lift; put $\alpha = \Delta_P(\xi)$ and $D =  |\langle

%\alpha \rangle|$. Suppose that for all $0 \neq x \in E(K)/PE(K)$,

%$\alpha$ and $\Delta_P(\xi+x)$ are independent. Then

%\[I(C) = P \cdot D. \]

%\end{theorem}

%\noindent Remark: The independence hypothesis is automatically

%satisfied when $E(K) = PE(K)$.

\begin{proof} As $x$ runs through $E(K)/PE(K)$, the elements $\xi + x$
run through all Kummer lifts of $\eta$.  For any Kummer lift
$\xi$, let $D = \#  \Delta_P(\xi) $.  Then $\Delta_{PD}(i(\xi)) =
D \Delta_P(\xi) = 0$, so that there is a rational divisor of
degree $PD$ on the corresponding torsor, and $I \leq PD$.
\end{proof}
\noindent
Concerning the inequality (\ref{FUNDINEQ}), Proposition
\ref{EASYPROP} asserts that the left hand side equals $1$ if and onlf if the
right hand side equals $1$.  When $P = p$ is prime, we have a
simple dichotomy: either $I/P = 1$ or $I/P = p$, so equality holds
in (\ref{FUNDINEQ}) when the period is prime, a fact which was
exploited in \cite{WCI}.  By a primary decomposition argument, we
also have equality when $P$ is squarefree.  It is not hard to see
that equality holding in (\ref{FUNDINEQ}) is equivalent to the
\emph{splitting} of the short exact sequence

\begin{equation}
\label{ISITSPLIT} 0 \ra \kappa^0(C/K) \ra \kappa(C/K) \ra Q \ra 0,
\end{equation}
where the last term $Q$ is cyclic of order $\frac{I}{P}$.  It is
natural to wonder whether this sequence \emph{always} splits. This
innocuous-looking question lies at the heart of the relationship
between the period, the index and the period-index obstruction
map, and it turns out to be surprisingly difficult.  We are
inclined to believe that the answer is in general negative.
However it is possible to show that equality holds for certain
specially constructed classes. In the proofs of the main theorems
we use Lichtenbaum-Tate duality to ensure equality, following
\cite{Sharif}. 

\subsection{The case of full level $P$ structure} \textbf{} \\ \\ \noindent
In this section we assume that that $E[P](\overline{K}) \subset E(K)$. By the theory of
the Weil pairing, the $P$th roots of unity $\mu_P$ are contained
in $K$. Fix a basis $(S,T)$ for $E[P]$ once and for all.  Note that this induces, via the Weil pairing, 
a basis for $\mu_P$ --- i.e., a specific primitive $P$th root of unity $\zeta = e_P(S,T)$.  After making this 
choice, we get an isomorphism

\begin{equation}
\label{FULLLEVELISO} \Phi: \Hp^1(K,\mu_P)\times\Hp^1(K,\mu_P)
\stackrel{\sim}{\ra} \Hp^1(K,E[P]).
\end{equation}
The composition of the cup product with the map $\mu_P\otimes\mu_P\rightarrow\mu_P$ given by
$\zeta^a\otimes\zeta^b\mapsto \zeta^{ab}$ gives a pairing \[\langle \ , \ \rangle_P:
\Hp^1(K,\mu_P)\times\Hp^1(K,\mu_P)\rightarrow\Hp^2(K,\mu_P)=\Br(K)[P], \]
the \textbf{level P norm residue symbol} (or \textbf{Hilbert
symbol}) \cite[p. 207]{CL}.
\\ \\
Via the canonical Kummer isomorphism $H^1(K,\mu_P) =
K^{\times}/K^{\times P}$, we may equally well view $\Phi$ and
$\langle \ , \ \rangle_P$ as maps defined on
$(K^{\times}/K^{\times P})^2$.

\begin{theorem} \label{THM2.7}
If $E[\Pstar] \subset E(K)$, then $ \Ob_P = \langle \ ,
\ \rangle_P$.
\end{theorem}
\noindent As a prelude to the proof, we consider the \textbf{special
theta group}.  Recall the theta group scheme $\sG_L$, where $L$ is
the class of $P[O]$. We found a homomorphism from $\sG_L$ to
$\Gl_P$. Let $\sS_L$ be the fiber product $\sG_L \times_K \Sl_P$,
where $\Sl_P \subset \Gl_P$ is the special linear group. Then we
have an exact sequence

\[ 0 \to \mu_P \to \sS_L \to E[P] \to 0, \]
where the maps are the restrictions of the maps in~\eqref{eq:theta}. If we identify $\Hp^2(K,\mu_P)$ with $(\Br K)[P]$, then the coboundary $\Hp^1(K, E[P]) \to \Hp^2(K,\mu_P)$ is the obstruction map. Let $\cbd:\Hp^0(K,E[P])\to \Hp^1(K,\mu_P)$ be the lower dimension coboundary. Define

\[ d:\Hp^1(K, E[P]) \to (\Br K)[P] \]
to be given by $d\xi(\sigma, \tau) = \cbd(\xi(\tau))(\sigma)$. (Note that since $E[P]$ is a trivial Galois module, each 
cohomology class in $\Hp^1(K,E[P])$ consists of a single cocycle.) Then

\begin{lemma}\label{Ob=hilb+d}
$\Ob = \langle \ , \ \rangle + d$.
\end{lemma}

\begin{proof}
As mentioned above, we have earlier shown \cite[Thm. 6]{WCI} that
$\Ob-\langle \ , \  \rangle$ is a homomorphism of groups.
Therefore it suffices to prove the claim for any subset of
$\Hp^1(K,E[n])$ which generates the group. We will consider the
subset given by the images of $\Hp^1(K,\Z/n\Z)$ induced by the two
maps $(1\mapsto S)$ and $(1\mapsto T)$. By symmetry, it suffices
to consider the case $(1\mapsto S)$ only.
Let $a\in\Hom(\mathfrak{g}_K, \Z/n\Z)$, and let $\xi$ be the image of $a$
under the map $(1\mapsto S)$. Clearly $\langle\xi\rangle=0$. Map
$S$ down to $PGL_n(K)$, then lift to an element $M_S$ in
$SL_n(\Kb)$. We set $M_{aS}=M_S^a$. Note that since $\det M_S=1$
and $P$ has order $n$, we must have $M_S^n=I$. Then

\begin{align*}
(\Ob \xi)(\sigma, \tau) & = M_S^{a(\sigma)}\sigma M_S^{a(\tau)} M_S^{-a(\sigma \tau)} \\
    & = M_S^{a(\sigma)} a(\tau)\cdot \cbd(S)(\sigma) M_S^{a(\tau)} M_S^{-a(\sigma \tau)} \\
    & =  a(\tau)\cdot \cbd(S)(\sigma) \\
    & = \cbd(\xi(\tau))(\sigma) \\
    & = d\xi(\sigma, \tau)
\end{align*}
The second equality follows from the fact that $\cbd(S)(\sigma)=\sigma M_S M_S^{-1}$.
\end{proof}

\begin{lemma}\label{2d=0}
$2d = 0$.
\end{lemma}

\begin{proof}
It suffices to show that $2\cbd = 0$. Let $\iota$ be the group inverse map on $E[P]$. 
According to~\cite[p.~308]{Mumford}, $\iota$ extends to a map on the theta group $\sG_L$ which acts as the 
identity on $\G_m$. We restrict $\iota$ to $\sS_L$. By the functoriality of $\cbd$, if $x\in\Hp^0(K,E[P]) = E[P]$, 
then $\cbd\circ\iota (x) = \cbd (x)$. But $\cbd\circ\iota(x) = \cbd(-x) = -\cbd(x)$, which proves the claim.
\end{proof}

\begin{proof}[Proof of Theorem \ref{THM2.7}]
If $P$ is odd, then $\Hp^1(K,\mu_P)$ has trivial 2-torsion. Therefore Lemma~\ref{2d=0} implies that 
$d=0$. By Lemma~\ref{Ob=hilb+d}, the conclusion follows.
\\ \\
Now suppose $P$ is even. According to~\cite[p.~310]{Mumford},
there is a map $\eta_2:\sG_{L^2} \to \sG_L$ which, upon
restriction to the subgroup schemes $\sS_L$ and $\sS_{L^2}$,
induces the commutative diagram

\begin{displaymath}
\xymatrix{\Hp^0(K,E[2P]) \ar[r]^\cbd \ar[d]^{[2]} & \Hp^1(K,\mu_{2P})\ar[d]^{[2]}\\
\Hp^0(K,E[P]) \ar[r]^\cbd & \Hp^1(K,\mu_{P})}.
\end{displaymath}
By the proof of Lemma~\ref{2d=0}, $[2]\circ\cbd$ is the zero map. Therefore $\cbd\circ[2]$ is zero. The 
hypothesis $E[2P]\subset E(K)$ implies that the left hand map above is surjective, and therefore the lower map $\cbd$ 
is zero. By Lemma~\ref{Ob=hilb+d}, the result follows.
\end{proof}

\noindent %Remark: This result has a complicated history.  It was

%claimed by O'Neil with $P$ instead of $\Pstar$ \cite[Prop.

%3.4]{O'Neil}, but this turns out to be false when $P = 2$: see

%\cite[$\S 3.2$]{WCI} for a discussion.  However, it is very close

%to being true! First, the quadratic map $\Delta$ decomposes into

%the sum of a quadratic form $\Delta_2$ and a linear form

%$\Delta_1$ (a general property of quadratic maps $q$ with $q(0) =

%0$), and the first author showed \cite[Prop. 7]{WCI} that the

%quadratic part of $\Delta$ is precisely the norm residue symbol.

%Moreover, O'Neil later pointed out that her argument was correct

%when $P$ is odd [...], so the novelty here is the assertion that

%when $P$ is even, taking full level $\Pstar = 2P$ structure is enough

%to trivialize the linear part of the obstruction map.

\section{Proofs of Theorems \ref{MT1}, \ref{MT2} and \ref{MT3}}
\noindent We first remind the reader of a standard trick: in all
work on the period-index problem it suffices to treat the case
where the period $P$ is a prime power $P = p^a$. Indeed, if a
class $\eta \in H^1(K,E)$ (or any other Galois cohomology group,
for that matter) has period $P = p_1^{a_1} \cdots p_r^{a_r}$, then
putting $\eta_i = \frac{P}{p_i^{a_i}} \eta$, one easily checks
that $\eta = \sum_{i=1}^r \eta_i$ and that $I(\eta) =
\prod_{i=1}^r I(\eta_i)$ (i.e., the index of $\eta$ is the product
of the indices of the classes $\eta_i$). The advantage of reducing
to the case $P = p^a$ is that then the index $I = p^{b}$ for $a
\leq b \leq 2a$ and then for any $D = p^c$, if the index $I$ is
less than $DP$, then indeed $I$ is a proper divisor of $DP$.

\subsection{Conditions on prime ideals and their generators}

\noindent Several times in the proofs we will be choosing pairs of prime ideals $v$, $v'$ of $\OO_K$ 
so as to satisfy certain conditions. Let us first say that a prime ideal $v$ of $K$ is \textbf{bad} (for $E$ and 
$P = p^a$) if $v$ is Archimedean, $v$ divides $p$, or $E$ has bad reduction at $v$, and is \textbf{good} otherwise. All 
but finitely many primes are good.
\\ \\
The other conditions we will impose on $v$ and $v'$ can all be achieved by using the Chebotarev density theorem. 
The conditions are
\begin{itemize}
    \item[(SC1)] The primes $v = (\pi)$  and $v'=(\pi')$ are principal, with totally positive generators $\pi$ and $\pi'$.
    \item[(SC2)] All elements of $E(K)$ are $P$-divisible in $E(K_{v})$.
    \item[(SC3)] The generators $\pi$ and $\pi'$ lie in $K_w^{\times P}$ for all bad primes $w$.
    \item[(SC4)] The order of the image of $\pi'$ in $K_v^\times/K_v^{\times P}$ is $P$.
\end{itemize}

\begin{lemma}\label{lem:choose-primes}
There exist infinitely many pairs of primes $v=(\pi)$ and $v'=(\pi')$ satisfying conditions (SC1)--(SC4).
\end{lemma}

\begin{proof}
Condition (SC1) is equivalent to $v$ and $v'$ splitting completely in the Hilbert class field of $K$. 
Condition (SC2) is equivalent to $v$ splitting completely in $K([P]^{-1}E(K))$, the field obtained by adjoining to $K$ 
all points $Q \in E(\overline{K})$ such that $[P]Q \in E(K)$. (Recall that $K([P]^{-1}E(K))$ is a finite abelian extension 
of $K$ unramified at the bad primes (e.g. \cite[p.194]{AEC}).) 
\\ \\
Let $\sm$ be the modulus given by the product of all bad primes $\sp$ and $P^2$. Then one can find $\pi$ and $\pi'$ as 
in (SC3) provided $v$ and $v'$ split completely in the ray class field for $K$ modulo $\sm$. For if $v$ splits completely, 
it has trivial Frobenius and, by class field theory, has a generator $\pi$ which is congruent to $1 \pmod{\sm}$. The 
condition follows from Hensel's Lemma.
\\ \\
Therefore, to satisfy conditions (SC1)--(SC3), we need $v$ and $v'$ to split completely in the abelian extension $F$ 
which is the compositum of the Hilbert class field of $K$, $K([P]^{-1}E(K))$, and the ray class field $K_\sm$.
\\ \\
Now we consider (SC4). Let $\alpha$ be a unit in $K_v$ which has order $P$ in $K_v^\times/K_v^{\times P}$. Let $F'$ be 
the ray class field with modulus $v$. By class field theory, the Galois group of $F'/K$ is isomorphic to the ideal class 
group with modulus $v$, $C_v$. In particular, if $v'$ and $(\alpha)$ lie in the same class in $C_v$, then $v'$ has a 
generator $\pi'$ which is congruent to $\alpha \pmod{v}$, and hence satisfies (SC4).
\\ \\
Thus, we have reduced conditions (SC1)--(SC4) to two splitting-type conditions in the abelian extensions $F$ and $F'$. 
It suffices to show that these splitting conditions are compatible, since then the Chebotarev density theorem shows there 
are infinitely many primes satisfying the conditions. 
\\ \\
The extension $F/K$ is unramified at $v$, while $F'/K$ is unramified outside $v$. Therefore $F\cap F'$ is contained 
in the Hilbert class field of $K$. Any $v'$ which lies in the same class as $(\alpha)$ in $C_v$ must be principal, 
and hence splits in $F\cap F'$. We conclude that the splitting conditions are compatible, which proves the lemma.
\end{proof}

\subsection{Proof of Theorem \ref{MT1}} \textbf{} \\ \\ \noindent
We assume in this section that $E$ has full level $\Pstar$-structure,
and maintain the setup of $\S 2.5$.  In particular, we have a
fixed isomorphism

\[ \Phi: (K^{\times}/K^{\times P})^2 \cong H^1(K,E[P]). \]
Let $v=(\pi)$ and $v'=(\pi')$ satisfy conditions (SC1)--(SC4). Put
\[\xi := \Phi(\pi^{P/D},\pi') \in H^1(K,E[P]), \] so by Theorem
\ref{THM2.7} we have
\[\Delta_P(\xi)  = \langle \pi^{P/D}, \pi' \rangle_P \in \Br(K).\]
Observe that $\Delta_P(\xi)$ is locally trivial away from $\pi$
and $\pi'$.  Indeed, by condition (SC3), the norm residue
symbol is trivial at the Archimedean places and at the places of
residue characteristic dividing $P$.  At all other places the norm
residue symbol is ``tame'' and hence vanishes locally at $w$ when
evaluated on a pair of $w$-adic units.
\\ \indent
Let $C$ be the genus one curve corresponding to the image $\eta$
of $\xi \in H^1(K,E)[P]$. Certainly the period of $\eta$ divides
$P$.  Suppose that the period of $\eta$ is less than $P$; then
(since $p^a \eta = 0$) it has period $P'$ for some proper divisor
$P'$ of $P$: $P'\xi = \iota_P(x)$. Then $\iota_P(x)$ is unramified
at $\pi'$ \cite[Prop. VIII.2.1]{AEC}, whereas $P'\xi =
(\pi^{PP'/D},(\pi')^{P'})$ is ramified at $\pi'$, a contradiction.
So $C$ has period $P$.  Moreover, by Proposition \ref{SHARPROP},

\[\Delta_{PD} i(\xi) = D \Delta_P(\xi) = D \langle \pi^{P/D}, \pi'
\rangle_P = \langle \pi^P,  \pi' \rangle_P = 0, \] 
so there exists a rational divisor of degree $PD$ on $C$ and $I(C) \ | \ PD$. \\
\indent Coming now to the heart of the matter, we suppose that the
index $I$ of $C$ strictly divides $PD$. Then, by Proposition
\ref{EASYPROP} there exists some lift $\nu$ of $\eta$ to
$H^1(K,E[I])$ such that $\Delta_I(\nu) = 0$. On the other hand,
the local-at-$\pi$ norm-residue symbol $\langle \pi^{P/D}, \pi'
\rangle_{P,\pi}$ has exact order $D$, since, by condition (SC4), the corresponding
central simple algebra trivializes over the Brauer group of an
extension $L/K_{v}$ if and onlf if $\pi'$ is a norm from the extension
$L(\pi^{\frac{1}{D}})/L$ if and onlf if $D \ | \ e(L/K)$. Therefore the
global norm residue symbol $\langle \pi^{P/D} ,\pi' \rangle_P =
\Delta_P(\xi)$ has order at least $D$; since $I/P < D$ we must
have
\[0 \neq (I/P) \cdot \Delta_P(\xi) = \Delta_I(j_{I/P}(\xi)). \]
For the remainder of the proof we shall abbreviate $j_{I/P}(\xi)$
to $j(\xi)$.  The classes $j(\xi)$ and $\nu \in H^1(K,E[I])$ are
both Kummer lifts of $\eta$ so there exists $x \in E(K)$ with
\[\iota_I(x) = \nu - j(\xi). \] 
Applying $\Delta$, we get
\[0 = \Delta_I(\nu) = \Delta_I(j(\xi)) + \Li(j(\xi),x). \]
Now recall that $(\pi)$ splits completely in $K([P]^{-1}E(K))$ by condition (SC2). 
This forces $E(K)$ to be divisible by $P$ in $E(K_v)$, and in particular $x \in PE(K_v)$.  Thus -- employing again 
the notation of (\ref{KUMMEREQ}) -- we have that 
$\iota(x)$ is locally trivial at $(\pi)$, hence so also is $\Li(j(\xi),x)$, implying that the restriction 
of $\iota(x)$ to $(\pi)$ is trivial.  It follows that the $(\pi)$-component of $\Li(j(\xi),x)$ and hence also
$\Delta_I(j(\xi))$ are trivial.  Thus $\Delta_I(j(\xi)) =
(I/P)\Delta_P(\xi)$ is locally trivial at all places except
possibly at $(\pi')$, and by the reciprocity law and Hasse
principle in the Brauer group of a local field this implies that
it is globally trivial---$\Delta_I(j(\xi)) = 0$---a
contradiction.
\\ \\
Finally, we claim that the image $\eta$ of $\xi$ under
$H^1(K,E[P]) \ra H^1(K,E)[P]$ is locally trivial away from $v$ and
$v'$.  First let $w$ be a bad prime.  Then, by construction, $\pi'
\in K_w^{\times P}$ so $\xi|_{K_w} = 0$; \emph{a fortiori} $\eta_w
= 0$.  Now suppose $w \neq v, v'$ is a good prime. Let $K_w^{\unr}$ be the maximal 
unramified extension of $K_w$. Recall that the restriction map $H^1(K_w,E)[P] \ra H^1(K_w^{\unr},E)[P]$
is injective \cite[Cor. 1]{LT}; this follows, for instance from
the triviality of WC-groups over finite fields together with the
fact that formation of the N\'eron model of a genus one curve
commutes with unramified base change.  Since
$K_w((\pi')^{\frac{1}{P}})/K_w$ is unramified, $\xi$ trivializes
over $K_w^{\unr}$.  But this implies that $\zeta |_{K_w^{\unr}} =
0$ and hence that $\eta |_{K_w} = 0$. This completes the proof of
Theorem 1.

\subsection{Proof of Theorem~\ref{MT2}: preliminaries} \textbf{} \\ \\ \noindent
First, we wish to reduce to Theorem \ref{MT1}, i.e., to the case where
$E[\Pstar]$ has trivial Galois module structure.  To this end we
introduce the splitting field $K_P = K(E[\Pstar])$ of the
$\Pstar$-torsion. We will construct classes $\theta_n$ in $\Hp^1(K_P,E[P])$ in a similar 
manner as in the proof of Theorem~\ref{MT1}, then we will set $\xi_n = \cores_{K_P/K} \theta_n$, 
and let $\eta_n$ be the image of $\xi_n$ in $\Hp^1(K,E)$. In order to prove that the $\eta_n$ have the 
right properties, we will need to compute $\res_{K_P/K} \xi_n = \res \circ \cores \theta_n$ explicitly. 
\\ \\
In the following, let $\langle, \rangle$ denote the $P$-Hilbert symbol on $\kummer{K_P}{P}$.

\subsection{Proof of Theorem \ref{MT2}: choosing pairs of primes} \textbf{} \\ \\ \noindent
In this section, we choose pairs of primes in a similar manner as in Lemma~\ref{lem:choose-primes}. The main difference 
is that we wish to choose an infinite sequence of pairs of primes $v_i, v_i'$ in $K_P$ inductively.  We will require 
conditions which are similar, and in some cases identical, to (SC1)--(SC4). These conditions are as follows:
\begin{itemize}
    \item[(SC1$'$)] The primes $v_i = (\pi_i)$ and $v'_i = (\pi'_i)$ are principal, with totally positive generators $\pi_i$ and $\pi'_i$.
    \item[(SC2$'$)] Let $\tilde{v}$ and $\tilde{v}'$ be primes of $K$ lying below $v_i$ and $v'_i$ respectively 
(for fixed $i$). Then all elements of $E(K)$ are $P$-divisible in $E(K_{\tilde{v}})$ and in $E(K_{\tilde{v}'})$.
    \item[(SC3$'$)] The generators $\pi_i$ and $\pi_i'$ lie in $(K_P)_w^{\times P}$ for all bad primes $w$ and for  $w=v_j$, $v_j'$ where $j<i$. 
    \item[(SC4$'$)] The order of the image of $\pi'_i$ in $(K_P)_{v_i}^{\times}/(K_P)_{v_i}^{\times P}$ is $P$. Additionally, $\sigma \pi'_i$ lies in $(K_P)_{v_i}^{\times P}$ for all nontrivial $\sigma \in \Gal(K_P/K)$.
    \item[(SC5$'$)] The primes $\tilde{v}, \tilde{v}'$ split completely in $K_P$.
\end{itemize}

\begin{lemma}\label{chooseq}
There exist $v_i = (\pi_i), v'_i = (\pi'_i)$ satisfying conditions (SC1$'$)--(SC5$'$).
\end{lemma}

\begin{proof}
We argue inductively: suppose that we have chosen $v_j, v'_j$ for $j<i$. Let $\sm$ be the modulus given 
by the product of all bad primes in $K$, $P^2$, and all $\sigma v_j$ and $\sigma v'_j$ for 
$j<i$, $\sigma \in \Gal(K_P/K)$. Let $F$ be the compositum of $K_P([P]^{-1}E(K))$ and the $\sm$-ray class field 
of $K_P$. Note that $\sm$ is rational over $K$, so $F$ is Galois over $K$. As before, $F$ is an abelian extension of 
$K_P$. By the Chebotarev density theorem, there exists a prime $\tilde{v}$ of $K$ which splits completely in $F$. Let 
$v_i$ be any prime of $K_P$ which lies over $\tilde{v}$. Then, provided (SC5$'$) holds, the same reasoning as in 
Lemma~\ref{lem:choose-primes} shows that $v_i$ satisfies all the conditions. (We need (SC5$'$) only for condition 
(SC2$'$), for otherwise we know only that $E(K_P)$ is $P$-divisible in $E((K_P)_{v_i})$.)
\\ \\
For simplicity, write $v$ in place of $v_i$. Let $\beta$ be a unit in $(K_P)_v$ which has order $P$ in 
$(K_P)_v^\times/(K_P)_v^{\times P}$. By the Chinese Remainder Theorem, there exists $\alpha\in K_P$ such that
\begin{align} \label{betacong}
\alpha & \equiv \beta \pmod{v} \notag \\
\alpha & \equiv 1 \pmod{\sigma v} \quad \forall \sigma \in \Gal(K_P/K),\, \sigma\neq 1.
\end{align}
Let $F'$ be the ray class field for $K_P$ with modulus $\sm' = \prod \sigma v$. Again, $\sm'$ is rational over $K$, 
so that $F'$ is Galois over $K$. Let $C_{\sm'}$ be the class group for $K_P$ with modulus $\sm'$. The Artin reciprocity 
map  gives an isomorphism $C_{\sm'} \to \Gal(F'/K_P)$. Let $\gamma_{F'}$ be the image of $(\alpha)$ under this isomorphism.  
Since $F\cap F'$ is contained in the Hilbert class field of $K_P$ and $(\alpha)$ is principal, there exists $\gamma \in \Gal(FF'/K_P)$ such that $\gamma|_{F'} = \gamma_{F'}$ and $\gamma|_F$ is the identity. Since $FF'$ is Galois over $K$, we view $\Gal(FF'/K_P)$ as a subgroup of $\Gal(FF'/K)$. Let 
$[\gamma]$ be the conjugacy class of $\gamma$ in this larger Galois group.  By Chebotarev, there exists a prime 
$\tilde{v}'$ of $K$ such that any Frobenius associated to $\tilde{v}'$ in the extension $FF'/K$ lies in $[\gamma]$. Let 
$v'_i$ be a prime of $K_P$ lying over $\tilde{v}'$. By replacing $v'_i$ by a conjugate if necessary, we may assume that the Frobenius of $v'_i$ in the extension $FF'/K_P$ is precisely $\gamma$ (the extension here is abelian, so saying ``the'' Frobenius makes sense). By the same arguments as in Lemma~\ref{lem:choose-primes}, $v'_i$ 
satisfies the first three conditions. 
\\ \\
One sees that $\pi'_i \equiv \alpha \pmod{(\pi_i)}$, so that the order of $\pi'_i$ in 
$(K_P)_{v_i}^{\times}/(K_P)_{v_i}^{\times P}$ is $P$. Also, $\pi'_i \equiv 1 \pmod{(\sigma \pi_i)}$ 
for nontrivial $\sigma$, so that $\sigma \pi'_i \equiv 1 \mod{(\pi_i)}$. Therefore $v'_i$ satisfies condition (SC4$'$). 
\\ \\
Any Frobenius associated to $\tilde{v}'$ in the extension $K_P/K$ is trivial, so that $\tilde{v}'$ splits in 
$K_P$, thus satisfying (SC5$'$).
\end{proof}

\subsection{Proof of Theorem \ref{MT2}: corestrictions} 
\label{sec:corestriction-map} 
\textbf{} \\ \\ \noindent
As in the proof of Theorem~\ref{MT1}, a choice of basis for $E[P]$ yields an isomorphism

\[ \Phi:\kummer{K_{P}}{P} \to \Hp^1(K_{P}, E[P]).\]
Let $\theta_n$ be either $\Phi(\pi_n,\pi'_n)$ or $\Phi(\pi_n,1)$, i.e., we will need to consider both cases.
Let $\cores$ be the corestriction map
\[ \Hp^1(K_{P},E[P]) \to \Hp^1(K,E[P]), \]
and write $\xi_n = \cores \theta_n$. In order to prove Theorem~\ref{MT2}, we would like to compute 
$\Ob_P (\xi_n - \xi_m)$ as well as the period of $(\xi_n-\xi_m)$. To do this, we will instead compute the 
obstruction and period of $\res (\xi_n - \xi_m)$, where $\res$ is the restriction map

\[ \Hp^1(K,E[P]) \to \Hp^1(K_{P},E[P]).\]
Both $\res$ and $\cores$ are $\Z$-linear, so it will suffice to compute $\res\circ\cores(\Phi(\pi_n,1))$ and 
$\res\circ\cores(\Phi(1,\pi'_n))$. 
\\ \\
Let $\Nm\in {\rm End}(\Hp^1(K_{P}, E[P]))$ be given, on the level of cocycles, by
\[ \Nm(\theta)(\sigma) = \sum_{\overline{\gamma}\in\Gal(K_{P}/K)} \gamma \cdot \theta(\gamma^{-1}\sigma\gamma), \]
where $\gamma$ is a fixed lift of $\overline{\gamma}$ to $\mathfrak{g}_K$. Since $E[P]$ is rational over 
$K_P$, there is a unique cocycle in each cohomology class, so that $\Nm$ is well-defined as an endomorphism of 
$\Hp^1(K_P, E[P])$.

\begin{lemma}
\label{RESCORESLEMMA}
If $\theta \in \Hp^1(K_{P}, E[P])$, then $\res\circ\cores \theta = \Nm \theta$.
\end{lemma}

\begin{proof}
The lemma follows from the definition of $\cores$ on $\Hp^0(K_{P}, E[P])$ and dimension shifting; see for 
example~\cite[p.119]{CL}.
\end{proof}
\noindent
In the remainder of this section, we drop the subscript $n$.
\\ \\
Lemma \ref{RESCORESLEMMA} shows that $\res\circ\cores(\Phi(\pi,1)) = \Nm(\Phi(\pi,1))$. Unfortunately, 
$\Nm$ and $\Phi$ do not commute, as the Galois actions on $E[P]$ and $\mu_P\times\mu_P$ differ. The representation 
on $E[P]$ gives, with respect to our fixed basis, a homomorphism 

\begin{align*}
\Gal(K_{P}/K) &\to \Gl_2(\Z/P\Z) \\
\sigma &\mapsto M_\sigma = \left(\begin{array}{cc} 
i(\sigma) & j(\sigma) \\
k(\sigma) & \ell(\sigma) \end{array}\right).
\end{align*}
Then we have

\begin{proposition}
Let $\sigma \in \Gal(K_{P}/K)$ and $(a,b) \in \kummer{K_{P}}{P}$. Then
\[ \Phi(a,b)^\sigma = \Phi\left(\frac{M_\sigma}{\det M_\sigma} (\sigma a, \sigma b)\right), \]
where $M_\sigma (a, b)$ is given by the natural action of $\Gl_2(\Z/P\Z)$ on $\kummer{K_{P}}{P}$; 
that is, $M_\sigma(a,b) = (a^{i(\sigma)}b^{j(\sigma}),a^{k(\sigma)}b^{\ell(\sigma)})$.
\end{proposition}

\begin{proof}
Our choice of basis for $E[P]$ gives rise to a group isomorphism
\[ \rho\colon  E[P] \to \mu_P \times \mu_P. \]
Define a $\Z[\Gal(K_{P}/K)]$-module $(\mu_P \times \mu_P)_\rho$ which,
as a $\Z$-module, is $\mu_P\times \mu_P$, but which possesses a
Galois structure making $\rho$ into a $\Gal(K_{P}/K)$-equivariant
map. In particular, if $(\zeta_1, \zeta_2) \in (\mu_P \times
\mu_P)_\rho$ and $\sigma\in \Gal(K_{P}/K)$, we have

\begin{align*}
\rho\circ\sigma\circ\rho^{-1} (\zeta_1, \zeta_2) & = \sigma (\zeta_1, \zeta_2)\\
& = M_\sigma (\zeta_1, \zeta_2).
\end{align*}
On the other hand, for $(\zeta'_1, \zeta'_2) \in \mu_P \times
\mu_P$ the Galois action is

\[ \sigma(\zeta'_1, \zeta'_2) = \det M_\sigma \cdot
(\zeta'_1, \zeta'_2), \]
where the action on the right is the diagonal action of $\Z/P\Z$.
\\ \\
Let $i: \mu_P \times \mu_P \to (\mu_P \times \mu_P)_\rho$ be the
canonical group isomorphism; it does not respect the
$\Gal(K_{P}/K)$-action. If $A$ is any $G_{K_{P}}$-module, write
$\Hp^1(A)$ for $\Hp^1(K_{P}, A)$. Then $i$ induces a map

\[ i_*:\Hp^1(\mu_P \times \mu_P) \to \Hp^1((\mu_P \times \mu_P)_\rho). \]
Let $M$ be either $(\mu_P \times \mu_P)_\rho$ or $\mu_P \times \mu_P$. 
Since in either case $M$ is a trivial $G_{K_{P}}$-module, the set of coboundaries 
$B^1(K_{P}, M)$ is zero, and so $\Hp^1(K_{P}, M) = Z^1(K_{P}, M)$, the set of
1-cocycles from $G_{K_{P}}$ to $M$. We can therefore identify
cohomology classes with cocycles in both cases.
\\ \\
Consider the commutative diagram

\begin{equation} \label{mu_P rho}
\xymatrix{\kummer{K_{P}}{P} \ar[r]^-\psi \ar[rd]_{\psi_\rho} & \Hp^1(\mu_P\times \mu_P) \ar[d]^{i_*} & \\
& \Hp^1((\mu_P\times\mu_P)_\rho) \ar[r]_-\lambda & \Hp^1(E[P])}.
\end{equation}
The horizontal maps are $\Gal(K_{P}/K)$-isomorphisms. The map $\lambda$ is induced by $(i\circ\rho)^{-1}$, and 
$\psi$ is the Kummer map. The diagonal map $\psi_\rho$ is $\psi\circ i_*$. Thus, $\Phi = \lambda\circ\psi_\rho$.  
Note that $\mathfrak{g}_K$ acts on all of the groups in~\eqref{mu_P rho} through its quotient $\Gal(K_{P}/K)$. Let $\gamma$ be an element of $\mathfrak{g}_{K_{P}}$ and 
$\sigma$ an element of $\mathfrak{g}_K$. Then

\begin{align} \label{kappa madness}
[\psi_\rho(a,b)]^\sigma (\gamma) & = [i_*\psi(a,b)]^\sigma(\gamma) \notag \\
& = \sigma [i(\psi(a,b)(\sigma^{-1}\gamma\sigma))] \notag \\
& = \sigma [i(\sigma^{-1}\sigma \psi(a,b)(\sigma^{-1}\gamma\sigma))] \notag \\
& = \sigma [i(\sigma^{-1} \psi(\sigma a, \sigma b)(\gamma))] \\
& = M_\sigma[(i(\det M_\sigma^{-1} \cdot \psi(\sigma a, \sigma b)(\gamma))] \notag \\
& = \frac{M_\sigma}{\det M_\sigma} [i(\psi(\sigma a,\sigma b)(\gamma))] \notag \\
& = \frac{M_\sigma}{\det M_\sigma} \psi_\rho(\sigma a, \sigma b) (\gamma) \notag
\end{align}
Applying $\lambda$ on both sides, we obtain the result.
\end{proof}

\begin{corollary}\label{cor:explicit-norm}
We have 
\[\Nm \Phi((a,b))    =  \Phi \left(\prod \sigma a^{i(\sigma)} \sigma b^{j(\sigma)},\prod \sigma a^{k(\sigma)} 
\sigma b^{\ell(\sigma)}\right), \]
where the product extends over all $\sigma \in \Gal(K_{P}/K)$.
\end{corollary}
\noindent
Let  $(c,d) = \Phi^{-1} \Nm \Phi (\pi,1)$ and $(c',d') = \Phi^{-1} \Nm  \Phi (1,\pi')$.

\begin{lemma} \label{a,b=P}
Let $v$ be the place of $K_P$ corresponding to $\pi$. Either $\#\langle c,d \rangle_v = P$ or $\#\langle cc', dd' \rangle_v = P$.
\end{lemma}

\begin{proof}
If $\#\langle c, d\rangle =P$, then we are done. So suppose that $\#\langle c, d\rangle <P$. In fact, since $P$ is a prime power, 
the order strictly divides $P$.
\\ \\
Expanding out the Hilbert symbol, we get
\[ \langle cc', dd'\rangle = \langle c, d \rangle + \langle c, d' \rangle + \langle c', d \rangle + \langle c', d' \rangle.\]
We have $\langle c, d' \rangle_v=\langle c', d' \rangle_v = 0$ since all are $v$-adic units. By our assumption at the 
start of the proof, $\langle c, d\rangle_v$ has order strictly dividing $P$. That leaves $\langle c', d\rangle_v$. By 
Corollary~\ref{cor:explicit-norm}, $c'=\pi' \cdot \displaystyle\prod_{\sigma\neq 1} (\sigma \pi')^{e_\sigma}$ for some 
integers $e_\sigma$. Our choice of $\pi'$ implies that $\pi'\equiv \alpha \pmod{(\pi)}$, where $\alpha$ was chosen to 
have order $P$ in $K_v^{\times P}$, while $\sigma \pi' \equiv 1 \pmod{(\pi)}$ for nontirivial $\sigma$ 
(see~\eqref{betacong}). Thus $c'\equiv \alpha \pmod{(\pi)}$. Therefore $K_v(c'^{1/P})/K$ is the unramified extension of 
degree $P$. (Equivalently, we may appeal to condition (SC4$'$).)
\\ \\
We now use similar reasoning as in the proof of Theorem~\ref{MT1} to see that $\langle c', \pi\rangle_v$ has order $P$. 
Since $v(d)=1$, the order of $\langle c', d\rangle_v$ is exactly $P$. This shows $\langle cc',dd' \rangle_v$ has exact order $P$.
\end{proof}
\noindent
If $\langle c,d\rangle$ has order $P$, let $\theta = \Phi(\pi,1)$, so 
that $\xi = \cores \theta$ satisfies $\res \xi = \Nm\Phi(\pi,1) = \Phi(c,d)$. Otherwise, 
let $\theta = \Phi(\pi,\pi')$, so that $\res \xi = \Phi(cc',dd')$. Let $(a,b)$ denote whichever pair 
we've chosen, $(c,d)$ or $(cc',dd')$.
\\ \\
Let us now reintroduce subscripts, so that 

\begin{align*}
\xi_n &= \cores \theta_n \\
&= \begin{cases} \cores \Phi(\pi_n, \pi'_n) \text{ or } \\
\cores \Phi(\pi_n, 1) \end{cases} \\
(a_n,b_n) &= \Phi^{-1} \res \xi_n.
\end{align*}

\begin{lemma}\label{lemma:sum a}
Let $0 \leq m < n$.  Then $\Delta_P(\res (\xi_m - \xi_n))$ has order $P$ at $v_m$.
\end{lemma}

\begin{proof}
Write $v$ for $v_m$. Since $E[\Pstar]\subset E(K_{P})$, the obstruction map can be computed using the Hilbert symbol.  Thus 
we wish to compute the order of

\[ \left\langle \frac{a_m}{a_n},\frac{b_m}{b_n} \right\rangle_v. \]
By the bilinearity of the Hilbert symbol, it suffices to compute

\[ \langle a_m, b_m\rangle_v -\langle a_m, b_n\rangle_v - \langle a_n, b_m \rangle_v + \langle a_n, b_n\rangle_v. \]
By Lemma~\ref{a,b=P}, the first term has order $P$. Since $a_n, b_m$ and $b_n$ are all units at $v$, the last two terms 
are zero. That leaves the term $\langle a_m, b_n\rangle_v$. By Corollary~\ref{cor:explicit-norm}, $b_n$ is a product 
of $\sigma \pi_n$ and $\sigma \pi'_n$. By condition (SC3$'$), these all lie in $K_v^{\times P}$. Therefore the second 
term is also zero. The Lemma follows.
\end{proof}

\subsection{Proof of Theorem \ref{MT2}: conclusion} \textbf{} \\ \\ \noindent
Let $C$ be the curve represented by the class $\xi:=\xi_i - \xi_j$ for some $i\neq j$. Clearly, $P(C) \ | \ P$.  If 
we can show that $I(C) = P^2$, then by (\ref{DIVEQ}) we must have $P(C) = P$.  
\\ \\
Since $E[P] \subset E(K_P)$ (and $E[2P] \subset E(K_P)$ when $P$ is even), the obstruction map on 
$\Hp^1(K_P, E[P])$ is given by the Hilbert symbol. By Lemma~\ref{lemma:sum a}, $\Delta_P(\res_{K_P/K} \xi)$ has order 
$P$ at $v_i$. Therefore $\Delta_P(\xi)$ has order $P$ at the prime $w$ satisfying $v_i\mid w$.
\\ \\
Suppose that $C$ has index $P\cdot D$ for some $D\mid P$. Then there exists some $\eta \in \Hp^1(K, E[PD])$ representing 
$C$ such that $\Delta_{PD}(\eta) = 0$. Let $j$ be the natural map $\Hp^1(K, E[P]) \to \Hp^1(K, E[PD])$. The classes 
$\eta$ and $j(\xi)$ represent the same curve $C$, so there exists some $x \in E(K)$ such that 
$\eta = j(\xi) + \iota_{PD}(x)$. Since $\Delta(\iota(x))=0$, by the remarks at the start of Section~\ref{3ASPECTS},

\[\Delta_{PD}(\eta) = \Delta_{PD}(j(\xi)) + \Li(\eta, x). \]
Recall that $\Li(\eta, x)$ is the Tate pairing. Let us consider this equality locally, at $w$. 
The left hand side is zero by hypothesis. By condition (SC2$'$), $x$ lies in $P\cdot E(K_w)$.  Since $P(C) \ | \ P$, 
the Tate pairing at $w$ is trivial. Hence $\Delta_{PD}(j(\xi))$ must be zero at $w$.  But by Proposition~\ref{SHARPROP},

\[ \Delta_{PD}(j(\xi)) = D \Delta_P(\xi). \]
We showed earlier that $\Delta_P(\xi)$ has order $P$ at $w$. Therefore $D=P$, and so $I(C) = P^2$.  
\\ \\
Let $\eta_i$ be the image of $\xi_i$ in $\Hp^1(K, E)$. It remains to show that $\res_v \eta_i = 0$ for $v\in S_K$. 
Recall that $\eta_i = \cores \Phi(\pi_i, 1)$ or $\cores \Phi(\pi_i,\pi'_i)$. For $w \mid v$ a place of $K_P$, the proof 
of Theorem~\ref{MT1} showed that the curves corresponding to $\Phi(\pi_i, 1)$ and $\Phi(\pi_i,\pi'_i)$ were trivial at $w$. 
But the corestriction map induces a homomorphism

\[ \oplus_{w\mid v} \Hp^1((K_P)_w, E) \to \Hp^1(K_v, E) \]
which proves that $\eta_i$ is trivial at $v$. This completes the proof of Theorem \ref{MT2}.

\subsection{Proof of Theorem \ref{MT3}}
\noindent Recall the following two ``classical'' instances of
period equals index.
\\ \\
(i) (Lang-Tate \cite{LT}) $F$ is the completion of a global field at a place
$v$, $E = \Jac(C)$ has good reduction, and $v$ does not divide the
period of $C$; and \\ (ii) (Cassels \cite{CasselsV}) $F$ is global and $C \in
\Sha(F,E)$.
\\ \\
Note that Lichtenbaum showed that $P=I$ for all genus one curves
defined over the completion of a global field. However, the result
of Lang and Tate, apart from being more elementary, is also more precise: 
they show also that an finite extension field $F'/F$ splits a genus one curve
$C_{/K}$ if and onlf if the period $P$ of $C$ divides the relative
ramification index $e(F'/F)$.  This will be used in the proof.
\\ \\
Take $S$ to be the union of the infinite places, the finite places
which divide $P$ and the places of bad reduction for $E$.  Let
$\{\eta_i\}_{i=0}^{\infty}$ be the sequence of places constructed
in \ref{MT2}.  We will show that for any positive integer $r$,
there exists a degree $P$ field extension $L/K$ such that the
classes are pairwise distinct, locally trivial, and of period $P$.
\\ \\
Indeed, let $S_r = \bigcup_{i=1}^r \supp(\eta_i)$.  We have $S_r
\cap S = \emptyset$, so that each $v_i \in S_r$ is a finite place
of good reduction for $E$ and residue characteristic prime to $P$.
\\ \\
For each $v_i \in S_r$, let $L_{i}/K_{v_i}$ be a totally ramified
extension of degree $P$.  There exists a degree $P$ global extension $L = L(r)$ of $K$ such
that for all $v_i \in S_r$, $L \otimes_K K_{v_i} \cong L_i$.\footnote{This is a standard weak approximation / Krasner's 
Lemma argument: c.f. \cite[p. 2]{WCI}.}  By the results of Lang and Tate cited above, $\eta_i|_{L}$ is locally
trivial.  Moreover, since $\eta_i = \eta_i - \eta_0$ has index
$P^2$ and $L/K$ is a degree $P$ extension, $I(\eta_i|_{L}) \geq
P$.  But on the other hand, by (ii) above, $I(\eta_i|_{L}) =
P(\eta_i|_{L}) \ | \ P(\eta_i) = P$, so for all $i$, $1 \leq i
\leq r$, $\eta_i|_{L}$ has period and index equal to $P$.  \\ \\
The only worry is that their restrictions are not
distinct.  But suppose that $\eta_i|_{L} = \eta_j|_{L}$.  Then
$\eta_i-\eta_j$ would lie in the kernel $\res_{L}$.  This would
imply that $I(\eta_i-\eta_j) \ | \ P$, which we have arranged not
to be the case.

\subsection{Remarks about ramification} \textbf{} \\ \\ \noindent
The proof of Corollary 3 differs from that of \cite[Theorem 1]{WCI} in that we explicitly make use of
extensions $L/K$ that are ramified at many primes.  Given our
strategy of proof, this is unavoidable: using the result (i) of
Lang-Tate cited above, the number of order $P$ elements in
$\res_L(\Hp^1(K,E)) \cap \Sha(L,E)$ can be bounded in terms of the
number of ramified primes of $L/K$.  It is interesting to ask
whether this same boundedness result holds for order $P$ elements
in $\Sha(L,E)$, and conversely, whether the number of order $P$
elements of $\Sha(L,E)$ necessarily approaches infinity with the
number of ramified primes.
\\ \\
Both of these questions have affirmative answers when $P = 2$,
according to work of H. Yu \cite{Yu}.  Given a quadratic extension
$L/K$, Yu computes the order of the kernel and cokernel of the
natural map $\Sha(K,E) \oplus \Sha(K,E^{\chi}) \ra \Sha(L,E)$;
here $E^{\chi}$ is the twist of $E_{/K}$ by the quadratic
character $\chi$ of $L/K$.  In particular, one can deduce
Theorem \ref{MT3} for $P = 2$ from H. Yu's work, with one
caveat: his analysis is conditional on the finiteness of
$\Sha(K,E)$. That the existence of an infinite subgroup of
$\Sha(K,E)$ would hamper our ability to show that $\Sha(L,E)[2]$
is large is somewhat curious, but seems to be the true state of affairs.
\\ \\
The consistency of Theorem \ref{MT3} with the results of \cite{Yu}
might thus be regarded as some confirmatory evidence for the
finiteness of Shafarevich-Tate groups. How seriously such evidence
ought to be taken is, of course, up to the reader to decide.

\section{Further problems}
\noindent Whereas in $\S 1.3$ we looked into the history of the period-index problem 
for genus one curves, in this final section we wish to look forward, by identifying and 
discussing some problems that remain open.
\\ \\
We assume that $E_{/K}$ is an elliptic curve and $P \ | \ I \ | \ P^2$ are positive integers.  However, we now 
allow the characteristic of $K$ to divide $P$.   
\begin{prob}
\label{PROB1} Find necessary and sufficient conditions on $E$ and
$K$ such that there exist infinitely many $\eta \in \Hp^1(K,E)$
such that $P(\eta) = P$, $I(\eta) = I$.  In particular, determine
whether this holds for every elliptic curve over an infinite,
finitely generated field.
\end{prob}
\noindent 
Of course our Theorem \ref{MT2} answers this question under certain, rather restrictive, 
hypotheses.  The case of $P = 2$ over a global field of characteristic different from $2$ is handled in 
\cite{P2}, whereas in \cite{Shahed08} an analogue of Theorem \ref{MT2} is proved under weaker hypotheses on the 
Galois module structure of $E[P^*]$: it is sufficient for $K$ to contain the $(P^*)$th roots of unity (i.e., that 
$\bigwedge^2 E[P^*]$ be a trivial Galois module) \emph{or} to contain a $K$-rational order $P^*$ cyclic subgroup scheme. 
\\ \\
In general, we have found it significantly easier to construct examples with $I = P^2$ rather than $I < P^2$.  
The following problem is motivated by a desire to show that this is the true state of affairs.

\begin{prob}
\label{PROB2} Show that ``most'' genus one curves of period $P$
have index $P^2$.
\end{prob}
\noindent To be sure, part of the problem is to find a precise
statement.  The interpretation we have in mind involves first constructing a ``versal'' parameter 
space $\mathcal{S}_P$ for curves of genus one and period $P$ over $K$.  In other words, $\mathcal{S}_P$ 
is a representable functor from the category of field extensions $L/K$ to the category of sets, together with functorial 
and surjective maps from $L/K$ to the set of isomorphism classes of genus one curves of period $P$.  For instance, 
$\mathcal{S}_2$ can be taken to be an open subset $U$ of $\mathbb{A}^8$, and then a versal family is the subset of 
$U \times \PP^3$ given by the system
\[ t_1 X^2 + t_2 Y^2 = Z^2,\ \  W^2 = t_3 X^2 + t_4 XY + t_5 XZ + t_6 Y^2 + t_7 YZ + t_8 Z^2.\]
Then one could construe Problem \ref{PROB2} either as saying that the generic point of $\mathcal{S}_P$ has index $P^2$, 
or as saying that the set of $K$-rational points of $\mathcal{S}_P$ for which $I < P^2$ is somehow sparse.%
\begin{prob}
\label{PROB3} Show that $\Delta_P(\Hp^1(K,E[P]))$ consists of Brauer
classes with period equals index (and perhaps even of cyclic
algebras).
\end{prob}
\begin{prob}
\label{PROB4} Construct an analogue of O'Neil's period-index obstruction map 
when $P$ is a power of the residue characteristic.
\end{prob}
\noindent If $K$ is perfect, then this is not very interesting:
taking $E[P]$ in the naive sense---i.e., as the sub-Galois module
of $E(\overline{K})$ consisting of elements killed by $P$---the
Kummer sequence
\[0 \ra E[P] \ra E(\overline{K}) \ra E(\overline{K}) \ra 0 \]
still holds, and since $\#E[P] \ | \  P$, every lift of $\eta \in
\Hp^1(K,E)[P]$ to $\xi \in \Hp^1(K,E[P])$ can be split by a degree
$P$ extension, i.e., $P = I$ in this case.  Alternately, it is
known that when $K$ is perfect of characteristic $p > 0$, $\Br(K)[p^{\infty}] = 0$. 
\\ \\
In the nonperfect case, multiplication by $P$ will not be
surjective on $E(\overline{K})$---remember that $\overline{K}$
denotes the \emph{separable} closure!---so that Kummer theory
is inapplicable. One can check that definitions (2) and (3) of the
period-index obstruction map go through in this case, although if
one insists on Galois cohomology definition (1) breaks down.
\\ \\
Nevertheless, Mumford went to some trouble to present a theory of
theta group \emph{schemes} which remains valid in all (odd)
characteristics: $\sG_L$ is in general an extension of the finite
flat group scheme $E[P]$ by $\Gm$.  One should still be able to
define a map $\Delta: \Hp^1(K,E[P]) \ra \Hp^2(K,\Gm)$, where the
cohomology is now \emph{flat} cohomology. What remains open in this 
case is the explicit computation of $\Delta$.\footnote{In fact we have some preliminary results in this 
direction, including a new ``cohomological symbol'' coming from supersingular elliptic curves.}  The relation with 
Problem \ref{PROB3} in this case seems especially interesting.
\begin{prob}
\label{PROB5} Decide whether the sequence (2) is always split.
\end{prob}
\noindent As we discussed in $\S 2.4$, this is an absolutely
fundamental question: it is equivalent to the tightest possible
relationship between the obstruction map $\Delta_P$ and the
period-index discrepancy $\frac{I}{P}$, since the conjecture holds
if and only if we have equality in~\eqref{FUNDINEQ}.  In this latter form
the problem is closely related to a question asked by O'Neil at
the end of $\S 2$ of \cite{O'Neil}.  

\begin{prob}
Explore relations with period-index problems for curves of higher
genus and for torsors of higher-dimensional abelian varieties.
\end{prob}
\noindent The prior work \cite{WCII} considers the case of torsor
under abelian varieties.  Some of the methods of the present work
could be adapted to the higher-dimensional case: for instance,
Theorem \ref{MT3} should hold for abelian varieties over global
fields which are principally polarized and have trivial Galois
action on their N\'eron-Severi group (with essentially the same
proof).  But the precise relation between the quantity $I/P$ and
the period-index obstruction map is, as yet, more mysterious in the
higher-dimensional case. 
\\ \\ 
Perhaps the most important open problem is to relate the period-index problem on a curve $C$
of higher genus $g$ to the period-index problem on its Jacobian
abelian variety.  In particular, can $I(C)/P(C)$ be computed via
some cohomological obstruction map?

\end{document}